\documentclass[a4paper, 12pt]{amsart}

\usepackage{amssymb}
\usepackage{amsthm}
\usepackage{amsmath}
\usepackage[mathscr]{eucal}
\usepackage{comment}
\usepackage{graphicx}

\theoremstyle{plain}
\newtheorem{thm}{Theorem}[section]
\newtheorem{prop}[thm]{Proposition}
\newtheorem{cor}[thm]{Corollary}
\newtheorem{lem}[thm]{Lemma}

\theoremstyle{definition}
\newtheorem{df}{Definition}[section]

\theoremstyle{remark}
\newtheorem{rmk}{Remark}[section]
\newtheorem*{ac}{Acknowledgements}

\newcommand{\zz}{\mathbb{Z}}

\newcommand{\rr}{\mathbb{R}}

\DeclareMathOperator{\card}{Card}

\newcommand{\grsp}{\mathscr{M}}
\DeclareMathOperator{\grdis}{\mathcal{GH}}

\DeclareMathOperator{\hdis}{\mathcal{HD}}

\DeclareMathOperator{\tdim}{\dim_{T}}

\newcommand{\deltasecond}{H^{\times}}

\newcommand{\deltasan}{\Lambda}
\DeclareMathOperator{\pmet}{PMet}

\newcommand{\sqsq}[1]{\mathbf{#1}}

\newcommand{\disdis}{R}
\newcommand{\ctree}{\Upsilon}
\newcommand{\cat}{\mathrm{CAT}}

\newcommand{\conset}{\mathscr{C}}
\newcommand{\pathset}{\mathscr{P}}
\newcommand{\geoset}{\mathscr{G}}
\newcommand{\catset}{\mathscr{Z}}

\newcommand{\grsptwo}{\mathscr{PM}}
\newcommand{\catsettwo}{\mathscr{PZ}}

\newcommand{\qqcube}{\mathbf{C}}

\DeclareMathOperator{\ppgrdis}{\widetilde{\mathcal{GH}_{*}}}
\DeclareMathOperator{\pgrdis}{\mathcal{GH}_{*}}

\newcommand{\topemb}{\rho}
\makeatletter
\@addtoreset{equation}{section}

\makeatother

\makeatletter
\@namedef{subjclassname@2020}{%
  \textup{2020} Mathematics Subject Classification}
\makeatother

\begin{document}
\title[Continua]
{
Continua 
in 
 the Gromov--Hausdorff space
 }

\author[Yoshito Ishiki]
{Yoshito Ishiki}
\address[Yoshito Ishiki]
{\endgraf
Photonics Control Technology Team
\endgraf
RIKEN Center for Advanced Photonics
\endgraf
2-1 Hirasawa, Wako, Saitama 351-0198, Japan}

\email{yoshito.ishiki@riken.jp}


\maketitle

\begin{abstract}
We first prove that 
for all compact metrizable spaces, 
there exists a topological embedding of the compact metrizable space 
into each of the sets of compact metric spaces which are 
 connected, path-connected, geodesic, or CAT(0),  in the Gromov--Hausdorff space with finite prescribed values. 
As its application, we show that 
the sets prescribed above  are path-connected and
 their 
  non-empty open subsets have infinite topological dimension.
By the same method, we also prove that the set of all proper CAT(0) spaces 
is path-connected and 
its  non-empty open subsets have   infinite topological dimension  with respect to 
the  pointed Gromov--Hausdorff distance. 
\end{abstract}

\section{Introduction}\label{sec:intro}
In \cite{Ishiki2021branching}, 
the author    constructed  
uncountable (cardinality of the continuum)
branching geodesics of 
the Gromov--Hausdorff distance continuously  parameterized by a Hilbert cube,  
passing through or  avoiding 
 sets of all spaces satisfying  some of  
the doubling property, the uniform disconnectedness, 
the uniform perfectness, 
and  passing through 
 sets of 
all infinite-dimensional spaces 
and the set of all metric spaces homeomorphic to the Cantor set. 
This  construction implies 
that the sets described above
contain a Hilbert cube as a 
topological subspace, and hence 
they 
 have infinite topological dimension with respect  to  the Gromov--Hausdorff distance.

In \cite{Ishiki2021fractal}, 
by constructing topological embeddings of 
compact metrizable spaces
into the Gromov--Hausdorff space, 
the author  proved 
that 
the set of all compact metrizable spaces 
possessing 
prescribed  
topological dimension,  
Hausdorff dimension, 
packing dimension, 
upper box dimension, 
and Assouad dimension, 
and the set of all compact  ultrametric spaces
are  path-connected and
 have  infinite topological dimension.

In this paper, 
as
a 
development  
 of the author's papers 
\cite{Ishiki2021branching} and 
 \cite{Ishiki2021fractal}, 
we prove that 
the sets of all compact metric spaces which are   connected, path-connected, 
 geodesic, or CAT(0)  have 
 infinite topological dimension
as subsets of  the Gromov--Hausdorff space,  
and we also prove that the set of  proper CAT(0) spaces 
has   infinite topological dimension  with respect to 
the  pointed Gromov--Hausdorff distance. 
 In this paper, 
 the topological dimension means the 
 covering dimension. 
 Since we only consider separable metric spaces, 
 the topological dimension in this paper coincides with 
 the large and small inductive dimensions. 
For the details of  dimensions of topological spaces, 
we refer the readers to \cite{HW1948}, \cite{P1975}, 
\cite{Nagata1983}, and  \cite{Cdimension}.

Let 
$(X, d)$ 
be a metric space. 
In this paper, for $x, y\in X$, 
a  map
$\gamma:[0, 1]\to X$ 
is said to be  
a \emph{geodesic from $x$ to $y$} if 
$\gamma(0)=x$ and $\gamma(1)=y$, and 
for all 
$s, t\in [0, 1]$ 
we have 
$d(\gamma(s), \gamma(t))=|s-t|\cdot d(x, y)$. 
Note that if there exists a curve from $x$ to $y$ whose length is 
$d(x, y)$, then there exists a geodesic from $x$
to $y$ (see \cite[Chapter 2]{BBI}). 
A metric space is said to be a 
\emph{geodesic space} if 
for all two points, 
there exists a geodesic connecting  them. 
A geodesic space $(X, d)$ is said to be 
a \emph{$\cat(0)$ space} if,  for all geodesic triangles 
$\bigtriangleup$ in $(X, d)$ and for all $x, y\in \bigtriangleup$, 
 we have 
$d(x, y)\le d_{\rr^{2}}(\overline{x}, \overline{y})$, where 
$d_{\rr^{2}}$ is the $2$-dimensional Euclidean metric and $\overline{x}$ and $\overline{y}$ are comparison points of $x, y$ in a comparison triangle 
$\overline{\bigtriangleup}$ in $\rr^{2}$ of $\bigtriangleup$. 
For the details of $\cat(0)$ spaces, 
we refer the readers to \cite{BH1999}. 
For a metric space 
$(Z, h)$,  
and  for subsets 
$A$,  $B$ 
of 
$Z$, 
we define  
the \emph{Hausdorff distance 
$\hdis(A, B; Z, h)$}
of 
$A$ 
and
$B$ 
in 
$(Z, h)$ by the 
infimum of all $r\in (0, \infty)$ such that 
for all  $a\in A$ and $b\in B$ there exist $u\in A$ and 
$v\in B$ with $h(a, v)\le r$ and $h(b, u)\le r$. 
For  metric spaces 
$(X, d)$ 
and 
$(Y, e)$, 
the 
\emph{Gromov--Hausdorff distance} 
$\grdis((X, d),(Y, e))$ between 
$(X, d)$ 
and 
$(Y, e)$ 
is defined as 
the infimum of  all values  
$\hdis(i(X), j(Y); Z, h)$, 
where 
$(Z, h)$ 
is a metric space, 
 and 
$i: X\to Z$ 
and 
$j: Y\to Z$ 
are isometric embeddings. 
We denote by 
 $\grsp$  
 the  set of all isometry classes of
  non-empty compact metric spaces. 
The space 
$(\grsp, \grdis)$ is called   the 
\emph{Gromov--Hausdorff space}.
By abuse of notation, 
we represent an element of $\grsp$ as 
 a pair $(X, d)$ of a set $X$ and a metric $d$ rather than its isometry class. 
We denote by $\conset$, 
$\pathset$, $\geoset$, and $\catset$
the subset of $\grsp$ consisting of  all compact metric spaces which are
 connected, path-connected, 
geodesic, and  $\cat(0)$,  respectively.

Similarly to \cite{Ishiki2021fractal}, 
by constructing 
a 
topological embedding 
of compact metrizable  spaces, 
we prove that 
 $\conset$, 
$\pathset$, $\geoset$, $\catset$, and 
$\catsettwo$
are path-connected and 
everywhere infinite-dimensional.

Our first main result is the next theorem, 
which  is an  analogue of \cite[Theorem 1.3]{Ishiki2021fractal}.
In the proof of \cite[Theorem 1.3]{Ishiki2021fractal}, 
the author used a construction of metrics on 
 the direct sum spaces since that theorem 
 treats
  properties preserved by the direct sum. 
  We can not apply that method 
   using the direct sum 
  to the class of continua 
  since 
  the direct sum of continua is not connected,  especially, it is 
  not a continuum. 
 In contrast to  my previous paper \cite{Ishiki2021fractal}, 
to prove  Theorem \ref{thm:connectedemb}, 
 we use metrics on the product spaces. 
We will prove  Theorem \ref{thm:connectedemb}  in Section \ref{sec:topemb}. 
\begin{thm}\label{thm:connectedemb}
Let $n\in \zz_{\ge 1}$. 
Let $H$ be a compact metrizable space, and 
$\{v_{i}\}_{i=1}^{n+1}$ be $n+1$ different points in 
$H$. 
Let $\mathscr{S}$ be any one of 
$\conset$, 
$\pathset$, $\geoset$, and $\catset$. 
Let $\{(X_{i}, d_{i})\}_{i=1}^{n+1}$ be a 
sequence  in $\mathscr{S}$ satisfying  that 
$\grdis((X_{i}, d_{i}), (X_{j}, d_{j}))>0$ for all  distinct $i, j$. 
Then, 
there exists a topological embedding 
$\Phi: H\to \mathscr{S}$ such that 
$\Phi(v_{i})=(X_{i}, d_{i})$. 
\end{thm}

Applying  Theorem \ref{thm:connectedemb} to 
$H=[0, 1]^{\aleph_{0}}$, 
by the fact that all
non-empty
 open subsets of $[0, 1]^{\aleph_{0}}$ have 
infinite topological dimension, 
we obtain:
\begin{cor}\label{cor:con}
The sets 
$\conset$, 
$\pathset$, $\geoset$, and $\catset$
are  path-connected and  all their  non-empty open 
subsets  have infinite topological dimension. 
\end{cor}

\begin{rmk}
The sets $\conset$, 
$\geoset$, and $\catset$ are closed and nowhere dense 
in $\grsp$;
however, $\pathset$ is not closed. 
For example, 
the so-called topologist's sine curve 
$\{\, 
(x, y)\in \rr^{2}
\mid 
x\in (0, 1], y=\sin(1/x)
\, 
\}
\cup
\{0\}\times [0, 1]$
 is a connected and 
non-path-connected compact metric space,  which is a limit of  path-connected  spaces 
$X_{n}=\{\, 
(x, y)\in \rr^{2}
\mid 
x\in [2^{-n}, 1], y=\sin(1/x)
\, \}$
with respect to $\grdis$. 
\end{rmk}

Our second result is a non-compact  analogue of  Theorem 
\ref{thm:connectedemb}. We construct topological embeddings 
of compact metrizable spaces into the space of the pointed 
proper metric  space. 
A metric space is said to be \emph{proper} if 
 all its bounded closed subsets are compact. 
A pair of a metric space and its point is called 
a \emph{pointed metric space}. 
We denote by $\grsptwo$
the set of pointed  isometry classes of pointed proper metric spaces. 
We represent elements of $\grsptwo$ in a similar way to $\grsp$. 
Let $(X, d)$ be a metric space. 
For  $x\in X$ and for  $r\in (0,\infty)$, 
we denote by $B_{d}(x, r)$ the closed  ball centered at $x$ with 
 radius $r$. 
For a subset  $A$ of $X$ and $r\in (0, \infty)$, we denote by  $N_{d}(A, r)$ the set of all $x\in X$ such that there exists $a\in A$ with $d(x, a)<r$. 
Let $t\in (0, \infty)$. 
Let $(X, d, a)$ and $(Y, e, b)$ be pointed metric spaces. 
We denote by $X\sqcup Y$ the direct sum of $X$ and $Y$. 
A metric $h$ on $X\sqcup Y$ is said to be 
\emph{ $(t; a, b)$-admissible} 
if 
$h|_{X^{2}}=d$ and $h|_{Y^{2}}=e$ and 
$h(a, b)<t$ and 
$B_{h}(a, t^{-1})\subset N_{h}(Y, t)$ 
and 
$B_{h}(b, t^{-1})\subset N_{h}(X, t)$. 
We define a quantity $\ppgrdis((X, d, a), (Y, e, b))$ by 
the infimum of all $t\in (0, \infty)$ such that 
there exists a $(t; a, b)$-admissible metric $h$ on $X\sqcup Y$. 
We also define $\pgrdis((X, d, a), (Y, e, b))$ by 
\[
\pgrdis((X, d, a), (Y, e, b))=\min\left\{\ppgrdis((X, d, a), (Y, e, b)),\  \frac{1}{2}\right\}.
\] 
The function $\pgrdis$ is a metric on $\grsptwo$
(see \cite[Corollary 3.14]{herron2016gromov}). 
Note that 
the convergence with respect to $\pgrdis$ is 
equivalent to the ordinary pointed  Gromov--Hausdorff convergence (see \cite[Proposition 3.5]{herron2016gromov}). 
The metric $\pgrdis$ is called the \emph{pointed Gromov--Hausdorff distance}.

Our second  main result is the next theorem, which will be 
proven in Section \ref{sec:pointed}. 
We denote by $\catsettwo$ the set of all $\cat(0)$ spaces in $\grsptwo$. 
\begin{thm}\label{thm:cat}
Let $n\in \zz_{\ge 1}$. 
Let $H$ be a compact metrizable space, 
and 
$\{v_{i}\}_{i=1}^{n+1}$ be $n+1$ different points  in 
$H$. 
Let $\{(X_{i}, d_{i}, a_{i})\}_{i=1}^{n+1}$ be a 
sequence  in $\catsettwo$ such that 
$\pgrdis((X_{i}, d_{i}, a_{i}), (X_{j}, d_{j}, a_{j}))>0$ for all  distinct $i, j$. 
Then, 
there exists a topological embedding 
$\Phi: H\to \catsettwo$ such that 
$\Phi(v_{i})=(X_{i}, d_{i}, a_{i})$. 

\end{thm}

\begin{cor}\label{cor:con}
The set $\catsettwo$
is   path-connected and   all its   non-empty open 
subsets  have infinite topological dimension. 
\end{cor}

\begin{ac}
The author would like to thank Takumi Yokota  
for 
raising  questions,  
for the many stimulating conversations, and 
for the many  helpful comments. 
The author would
also like to thank the referee for helpful comments and suggestions.
\end{ac}

\section{Metric spaces parametrized by a Hilbert space}\label{sec:fam}
A topological  space is said to be a \emph{Hilbert cube} if 
it is homeomorphic to the countable power of 
the closed unit interval $[0, 1]$ of $\rr$.
To prove Theorems \ref{thm:connectedemb} and 
\ref{thm:cat}, 
in this section,  we construct a family of 
metric trees injectively and continuously parametrized by a Hilbert cube.

For a set $X$, 
a map $d: X\times X\to [0, \infty)$ is said to be 
a \emph{pseudo-metric} if 
$d$ satisfies the triangle inequality and satisfies 
that $d(x, x)=0$  and $d(x, y)=d(y, x)$ for all $x, y\in X$.
If 
a pseudo-metric $d$ satisfies that  $d(x, y)=0$ implies $x=y$, then $d$ is a metric. 
We denote by $\pmet(X)$ 
the set of 
all pseudo-metrics
on $X$. 
We define a metric  $\mathcal{D}_{X}$ on 
 $\pmet(X)$ by 
 $\mathcal{D}_{X}(d, e)=
 \sup_{x, y\in X}|d(x, y)-e(x, y)|$. 
 Note that $\mathcal{D}_{X}$ 
can  take the value $\infty$. 
Let $d\in \pmet(X)$. 
We denote by $X_{/d}$ the quotient set 
by the relation $\sim_{d}$ defined by 
$x \sim_{d} y$ if and only if $ d(x, y)=0$.  
We denote by $[x]_{d}$ the equivalence class of $x$ by 
$\sim_{d}$. We define a metric $[d]$ on $X_{/d}$ by 
$[d]([x]_{d}, [y]_{d})=d(x, y)$. 
The metric $[d]$ is well-defined. 
Remark that if $d$ is a metric, then 
$(X_{/d}, [d])$ is isometric to $(X, d)$. 
The following was proven in \cite[Corollary 2.13]{Ishiki2021fractal}. 
\begin{prop}\label{prop:metconti}
Let $T$ be a topological space whose all 
finite subsets are closed. 
Let $X$ be a set. 
If a map $h: T\to \pmet(X)$  is continuous and 
$h(t)$ is a metric for all $t\in T$ except finite points,  
then the map $F: T\to \grsp$ defined by 
$F(t)=(X_{/h(t)}, [h(t)])$ is continuous. 
\end{prop}

A family of metric trees defined  in this section is 
a connected  analogue of a family of 
 compact  metric spaces homeomorphic to the one-point compactification of the countable discrete space,  
 which is constructed in 
\cite[Definition 4.1]{Ishiki2021branching} and \cite[Definition 4.1]{Ishiki2021fractal}. 
\begin{df}
We define 
$\qqcube=\prod_{i=1}^{\infty}[2^{-2i}, 2^{-2i+1}]$. 
Note that every $\sqsq{a}=\{a_{i}\}_{i\in \zz_{\ge 1}}\in \qqcube$ satisfies $a_{i}<1$ and $a_{i+1}<a_{i}$ for all $i\in \zz_{\ge 1}$  and $\lim_{i\to \infty}a_{i}=0$. 
We define a metric $\tau$ on $\qqcube$ by 
$\tau(x, y)=\sup_{i\in \zz_{\ge 1}}|x_{i}-y_{i}|$. 
Then,  $\tau$ generates the topology which makes 
$\qqcube$  a Hilbert cube. 
\end{df}

\begin{df}\label{df:tree}
Let $\sqsq{a}=\{a_{i}\}_{i\in \zz_{\ge 1}}\in \qqcube$. 
We supplementally put $a_{0}=1$. 
Put 
$\ctree=\{(0, 0)\}\cup (0, 1]\times \zz_{\ge 0}$. 
To simplify our description, 
we represent an element $(s, i)$ of $\ctree$ as 
$s_{i}$. For example, $0_{0}=(0, 0)$, 
$1_{n}=(1, n)$, 
and $(1/2)_3=(1/2, 3)$. 
We define a metric $\disdis[\sqsq{a}]$ on
$\ctree$ by 
\[
\disdis[\sqsq{a}](s_{i}, t_{j})
=
\begin{cases}
a_{i}|s-t| & \text{ if $i=j$;}\\
a_{i}s+a_{j}t & \text{ otherwise.}
\end{cases}
\]
Note that $(\ctree, \disdis[\sqsq{a}])$ is compact and 
the space 
$(\ctree, \disdis[\sqsq{a}])$ can be considered as a 
metric subtree of the spider tree which is  the plane equipped with 
the radial metric (see \cite[Examples 1.6 and 1.8]{AB2010}). 
The metric  $\disdis[\sqsq{a}]$ is  constructed in a similar way to \cite[Examples 1.6 and 1.8]{AB2010} with dilated  edges,  and each $a_{i}$ is the scaling factor  of the  $i$-th edge. 
Since all metric trees are $\cat(0)$ (see \cite[(5) in Example 1.15, p.167]{BH1999}), 
the space $(\ctree, \disdis[\sqsq{a}])$ is 
a  $\cat(0)$ space. 
Note that 
even if $\sqsq{a}\neq \sqsq{b}$, 
the metrics 
$\disdis[\sqsq{a}]$ and $\disdis[\sqsq{b}]$ generate 
the same topology on $\ctree$. 
\end{df}

\begin{prop}\label{prop:qqinj}
Let $\sqsq{a}=\{a_{i}\}_{i\in \zz_{\ge 1}}$ and 
$\sqsq{b}=\{b_{i}\}_{i\in \zz_{\ge 1}}$ be in $\qqcube$. 
Let $K, L\in (0, \infty)$. 
If $(\ctree, K\cdot\disdis[\sqsq{a}])$ and 
$(\ctree, L\cdot \disdis[\sqsq{b}])$
are isometric to each other, 
then $\sqsq{a}=\sqsq{b}$. 
\end{prop}
\begin{proof}
Let $f: (\ctree, K\cdot\disdis[\sqsq{a}])\to 
(\ctree,  L\cdot \disdis[\sqsq{b}])$ be an isometry. 
The point 
$0_{0}$ of $\ctree$ is the unique point such that 
$\ctree \setminus\{0_{0}\}$ has infinitely many 
connected components. 
Therefore $f(0_{0})=0_{0}$. 
Let $A$ be the set of all points $p$ in $\ctree$ such that 
$\ctree\setminus \{p\}$ is connected. 
Then  $A=\{1_{i}\mid {i\in \zz_{\ge 0}}\}$, 
and $f(A)=A$. 
We put $A_{0}=A$ and 
$A_{n}=A\setminus \{1_{i}\mid i= 0, 1\dots, n-1\}$
for each $n\in \zz_{\ge 1}$. 
Then,  for each $n \in \zz_{\ge 0}$, 
the point $1_{n}\in \ctree$ is characterized 
as 
the unique argument of the  maximum  of the map 
$p\in A_{n}\mapsto \disdis[\sqsq{a}](p, 0_{0})$, i.e., 
the point
$1_{n}$ is 
the unique point 
satisfying that 
$\disdis[\sqsq{a}](1_{n}, 0_{0})=\max_{p\in A_{n}}\disdis[\sqsq{a}](p, 0_{0})$. 
Since $f$ is an isometry,  
by induction, 
 $f(1_{n})=1_{n}$ for all $n\in \zz_{\ge 0}$. Thus, 
 for each $n\in \zz_{\ge 0}$,  we have
\[
K\cdot a_{n}=K\cdot \disdis[\sqsq{a}](1_{n}, 0_{0}) 
=L\cdot \disdis[\sqsq{b}](1_{n}, 0_{0})=L\cdot b_{n}, 
\]
and hence $Ka_{n}=Lb_{n}$ for all $n\in \zz_{\ge 0}$. 
By $a_{0}=b_{0}=1$, 
we have $K=L$. 
 Therefore, we conclude that
$\sqsq{a}=\sqsq{b}$. 
\end{proof}
\begin{prop}\label{prop:qqconti}
For all  $\sqsq{a}, \sqsq{b}\in \qqcube$, 
we have 
$\mathcal{D}_{\ctree}(
\disdis[\sqsq{a}], \disdis[\sqsq{b}])\le 2\tau(\sqsq{a}, \sqsq{b})$. 
\end{prop}

\begin{proof}
Let $s_{i}, t_{j}\in \ctree$. 
If $i=j$, 
then we have 
\[
|\disdis[\sqsq{a}](s_{i}, t_{j})-\disdis[\sqsq{b}](s_{i}, t_{j})|
=|a_{i}-b_{i}||s-t|\le |a_{i}-b_{i}|\le \tau(\sqsq{a}, \sqsq{b}). 
\]
If $i\neq j$, 
then we have 
\[
|\disdis[\sqsq{a}](s_{i}, t_{j})-\disdis[\sqsq{b}](s_{i}, t_{j})|
\le |a_{i}-b_{i}|s+|a_{j}-b_{j}|t\le  2\tau(\sqsq{a}, \sqsq{b}). 
\]
This finishes the proof. 
\end{proof}

\section{Topological embeddings}\label{sec:topemb}
In this section,  we prove Theorem \ref{thm:connectedemb}. 
Before doing that, we prepare and review  the basic constructions and properties  of 
metrics. 

\subsection{Amalgamation and product  of  metrics}

\begin{prop}\label{prop:amal1}
Let $X$ and $Y$ be sets. 
Let 
$d\in \pmet(X)$ 
and 
$e\in \pmet(Y )$. 
Assume that there exists a point $p$ such that  
$X\cap Y=\{p\}$. 
We define 
a symmetric function 
$h: (X\cup Y)^2\to [0, \infty)$ 
by
\begin{align*}
	h(x, y)=
		\begin{cases}
		d(x, y) & \text{if $x, y\in X$;}\\
		e(x, y) & \text{if $x, y\in Y$;}\\
		d(x, p)+e(p, y) & \text{if $(x, y)\in X\times Y$. }
		\end{cases}
\end{align*}
Then, the following statements hold  true. 
\begin{enumerate}
\item 
The function $h$ is a pseudo-metric and satisfies  $h|_{X^2}=d$ and 
$h|_{Y^{2}}=e$. \label{item:metric1}
\item 
If $d$ and $e$ are metrics, then  so is $h$. \label{item:metric2}
\item 
If $(X, d)$ and $(Y, e)$ are 
geodesic (resp.~$\cat(0)$) metric spaces, 
then so is $(X\cup Y, h)$. \label{item:metric3}
\end{enumerate}
\end{prop}
\begin{proof}

The statements (\ref{item:metric1}) and (\ref{item:metric2}) in   the 
proposition  are deduced from \cite[Proposition 3.2]{Ishiki2020int}.
The  statement on $\cat(0)$ spaces  in   (\ref{item:metric3}) follows from  
\cite[Theorem 11.1, 
p.347]{BH1999}.  
It suffices to show that, 
if $(X, d)$ and $(Y, e)$ is geodesic spaces, 
for all  
$x\in X\setminus \{p\}$ and $y\in Y\setminus \{p\}$ 
there exists a geodesic between $x$ and $y$. 
By the assumption, there exists a geodesic from  
$x$ to  $p$ and  a geodesic from $p$ to  $y$. 
By the definition of $h$, 
we can apply 
\cite[Proposition 2.6]{memoli2021characterization}
to these two geodesics, and 
we can glue these two geodesics together at the point $p$. 
Then, 
we obtain a 
geodesic between $x$ and $y$ (see also \cite[Lemma 5.24, p.67]{BH1999}). 
\end{proof}

For two metric spaces $
(X, d)$ 
and 
$(Y, e)$, 
we denote by 
$d\times e$  
the 
$\ell^{2}$-product metric  
defined by 
$(d\times e)((x, y), (u, v))=
\sqrt{d(x, u)^{2} + e(y, v)^{2}}$. 

The next lemma can be found in 
\cite[(3) in Example 1.15, p.167]{BH1999}. 
\begin{lem}\label{lem:prodcat}
Let $(X, d)$ and $(Y, e)$ be metric spaces. 
If $(X, d)$ and $(Y, e)$ are geodesic spaces 
(resp.~$\cat(0)$ spaces), then 
so is $(X\times Y, d\times e)$. 
\end{lem} 

By the definition of $\cat(0)$ spaces, we obtain:
\begin{lem}\label{lem:dilcat}
Let $L\in (0, \infty)$.
If $(X, d)$ is  a geodesic space  (resp.~$\cat(0)$ space),  
then  so is $(X, L\cdot d)$. 
\end{lem}

\subsection{Arcwise-connectedness}

We say that a subset of a topological space is 
a \emph{topological arc} if it is homeomorphic to  $\rr$. 
A topological space $X$ is said to be 
\emph{arcwise-connected} if 
for all two points $p, q\in X$ with $p\neq q$, there exists a 
topological embedding  $f: [0, 1]\to X$ with 
$f(0)=p$ and $f(1)=q$. 
The following is deduced from \cite[Corollary 31.6]{W1970}, and related to the Hahn--Mazurkiewicz theorem. 
\begin{lem}\label{lem:patharc}
A Hausdorff topological space is 
 path-connected if and only if  it is
arcwise-connected. 
In particular, 
if a Hausdorff  space $X$ has at least two points 
and $X$ 
is a continuous image of  $\rr$, 
then $X$ contains a topological   arc as a subspace. 
\end{lem}

For a topological space $X$, 
we denote by $\tdim X$ the topological   dimension of 
$X$. 
To construct topological embeddings of compact metrizable spaces into 
the Gromov--Hausdorff space (especially, to  guarantee  injectivity), we use the following lemma. 
\begin{lem}\label{lem:productinterval}
Let $X$ and $Y$ be Hausdorff topological spaces
possessing no isolated points. 
Then any topological arc in  the product space 
$X\times Y$ 
 has no interior points.
\end{lem}
\begin{proof}
For the sake of contradiction, 
we suppose that 
there exists a subset $V$ of $X\times Y$
which is 
a topological arc with non-empty  interiors. 
Since every point of  $\rr$ has 
a neighborhood system consisting of 
open intervals, 
by extracting an open subset if necessary, 
we may assume that $V$ is open in $X\times Y$. 

Take a point  $(p, q)\in V$. 
Then, there exists  open subsets $A$, $B$  of 
$X$, $Y$, respectively, with $p\in A$ and $q\in B$
and $A\times B \subset V$. 
Since $V$ is a topological arc and open in $X\times Y$, 
there exists an open subset $W$ of $X\times Y$ such that 
$W$ is a topological  arc, 
and $(p, q) \in W$, and $W\subset A\times B$. 
Since $X$ and $Y$ have no isolated points,
and since the projections $\pi_{X}$ and $\pi_{Y}$  are
open maps, 
the sets $\pi_{X}(W)$ and $\pi_{Y}(W)$ contain
at least two points. 
Thus, by Lemma \ref{lem:patharc}, 
they contain topological  arcs. 
By this observation, and by 
 $\tdim\rr^{2}=2$, we have
$2\le \tdim(\pi_{X}(W)\times \pi_{Y}(W))$. 
By this inequality and 
 $\pi_{X}(W)\times \pi_{Y}(W)\subset A\times B$, 
we have $2\le \tdim V$. This contradicts 
$\tdim V=\tdim \rr=1$. 
\end{proof}

\subsection{Topological embeddings}
For every $n\in \zz_{\ge 1}$, 
 we denote by $\widehat{n}$
the set $\{1, \dots, n\}$. 
In what follows, we consider that 
the set $\widehat{n}$ 
is always equipped with the discrete topology. 

The next proposition is an analogue of 
\cite[Proposition 4.4]{Ishiki2021fractal}. 
Unlike the proof of  \cite[Proposition 4.4]{Ishiki2021fractal}, 
 we now  use a construction of  metrics on  product spaces. 
\begin{prop}\label{prop:embednm}
Let $\mathscr{S}$ be any one of 
$\conset$, 
$\pathset$, $\geoset$, and $\catset$.  
Let $n\in \zz_{\ge 1}$ and $m\in \zz_{\ge 2}$. 
Let $H$ be a compact metrizable spaces, 
and $\{v_{i}\}_{i=1}^{n+1}$ be 
$n+1$ many different points in $H$. 
Put $\deltasecond=H\setminus \{\, v_{i}\mid i=1, \dots, n+1\, \}$. 
Let $\{(X_{i}, d_{i})\}_{i=1}^{n+1}$ be a 
sequence of compact metric spaces in $\mathscr{S}$ satisfying  that 
$\grdis((X_{i}, d_{i}), (X_{j}, d_{j}))>0$ for all distinct $i, j$. 
Then there exists a continuous map 
$F: H\times \widehat{m}\to 
\mathscr{S}$
such that 
\begin{enumerate}
\item 
for all $i\in \widehat{n+1}$ and $j\in \widehat{m}$ we have 
$F(v_{i}, j)=(X_{i}, d_{i})$;
\item 
for all $(s, i), (t, j)\in \deltasecond\times \widehat{m}$ with $(s, i)\neq (t, j)$, we have 
$F(s, i)\neq F(t, j)$. 
\end{enumerate}
\end{prop}
\begin{proof}
Since the case where  $n=1$ and either $X_{1}$ or  $X_{2}$ is the one-point metric space follows from 
 cases of $n>1$, we may assume that $n>1$,  or 
$n=1$ and both   $X_{1}$ and  $X_{2}$ contain at least two points. 

We denote by 
$\deltasan(H, m)$ the quotient space 
of $H\times \widehat{m}$ in a such a way that  for each $i\in \widehat{n+1}$ we identify $m$ many  points 
$\{\, (v_{i}, j)\mid  j\in \widehat{m}\, \}$ 
as a single point. 
Since $\deltasan(H, m)$ is compact and metrizable, 
there exists a topological embedding 
$\topemb:\deltasan(H, m)\to \qqcube$ 
(this is the Urysohn Metrization  Theorem, see \cite{Kelly1975}). 
Let $\pi_{H,m}:H\times \widehat{m}\to \deltasan(H, m)$ be the canonical projection. 
Since every metrizable space is perfectly normal 
(see \cite[Proposition 4.18]{P1975}), 
 for each $i\in \widehat{n+1}$, there exists a continuous function 
$\zeta_{i}:H\to [0, 1]$ such that 
$\zeta_{i}^{-1}(0)=\{\, v_{j}\mid j\neq i\, \}$ and 
$\zeta_{i}^{-1}(1)=\{v_{i}\}$, and there exists 
 a continuous function $\xi: H\to [0, 1]$
 with 
$\xi^{-1}(0)=\{\, v_{i}\mid i=1, \dots, n+1\, \}$. 
We define 
$P=\prod_{i=1}^{n+1}X_{i}$. 
For each $s\in H$,  we define  
a pseudo-metric $E_{s}$ on $P$ by 
\[
E_{s}(x, y)=
\sqrt{\sum_{i=1}^{n+1}(\zeta_{i}(s)\cdot d_{i}(x_{i}, y_{i}))^{2}}. 
\] 
Note that $E_{s}$  is  a metric if and only if 
 $s\in \deltasecond$, and note that 
 if $s\in \deltasecond$, the metric $E_{s}$ 
 generates the product topology on 
 $P=\prod_{i=1}^{n+1}X_{i}$. 
 Since each $X_i$ is compact, 
 the map $W:H\to \pmet(P)$
 defined by $W(s)=E_{s}$ is continuous. 
 By the fact that the connectedness and path-connectedness are 
 invariant under the topological product, 
 and 
 by Lemmas \ref{lem:prodcat} and \ref{lem:dilcat}, 
 if $s\in \deltasecond$, 
 the metric space $(P, E_{s})$ is in $\mathscr{S}$. 
 Take $p\in P$. 
 We identify $p$ with $1_{0}\in \ctree$, and 
 we consider $P\cap \ctree=\{p\}$. 
We put 
$Z=P\cup \ctree$. 
For each $(s, k)\in H\times \widehat{m}$,   we define a symmetric function  $D_{s, k}$ on $Z\times Z$ by 
\begin{align*}
	D_{s, k}(x, y)=
		\begin{cases}
		E_{s}(x, y) & \text{if $x, y\in P$;}\\
		\xi(s)\disdis[\topemb\circ \pi_{H, m}(s, k)](x, y) & \text{if $x, y\in \ctree$;}\\
		E_{s}(x, p)+\xi(s)\disdis[\topemb\circ \pi_{H, m}(s, k)](p, y) & \text{if $(x, y)\in P\times \ctree$. }
		\end{cases}
\end{align*}

Proposition  \ref{prop:amal1} implies that 
 $D_{s, k}$ is a pseudo-metric on $Z$ for all 
 $(s, k)\in H\times \widehat{m}$. 
We see that
$D_{s, k}$ is  a metric if and only if 
 $s\in \deltasecond$. 
 We also  see that 
 for all 
 $i\in \widehat{n+1}$ and  $k\in \widehat{m}$,   
  the quotient metric  space 
$\left(Z_{/D_{v_{i}, k}}, [D_{v_{i}, k}]\right)$ is isometric to 
$(X_{i}, d_{i})$. 
Note that if $s\neq v_{i}$ for any $i\in \widehat{n+1}$, 
the quotient metric space 
$\left(Z_{/D_{s, k}}, [D_{s,  k}]\right)$ is 
isometric to 
$\left(Z, D_{s, k}\right)$. 
By Proposition  \ref{prop:amal1},  
we have 
$(Z, D_{s, k})\in \mathscr{S}$ for all 
$(s, k)\in \deltasecond\times \widehat{m}$. 
We define  
$F: H\times \widehat{m}\to \mathscr{S}$ 
 by 
\[
F(s, k)=
\begin{cases}
(X_{i}, d_{i}) & \text{if $s=v_{i}$ for some $i\in \widehat{n+1}$;}\\
\left(Z, D_{s, k}\right) & \text{otherwise.}
\end{cases}
\]
Then the condition (1) is satisfied. 
Since the map $W:H\to \pmet(P)$ is continuous, 
by the definition of $D_{s, k}$, and by 
Proposition \ref{prop:qqconti}, 
the map 
$J:H\times \widehat{m}\to \pmet(Z)$ defined by 
$J(s, k)=D_{s, k}$ is continuous. 
Therefore, 
by Proposition  \ref{prop:metconti}, 
the map $F$ is continuous. 

Next we prove the condition (2). 
For a metric space $(S, h)$, 
we denote by  $\mathcal{A}(S, h)$ the closure of 
the set of all  points possessing neighborhood systems 
consisting of  topological  arcs. 
Note that if metric spaces $(S, h)$ and $(S', h')$ are isometric to each other, then so are 
$\mathcal{A}(S, h)$ and $\mathcal{A}(S', h')$. 
From  the assumption  that $n>1$, or $n=1$ and 
both $X_{1}$ and $X_{2}$ contain at least two points, 
it follows that   $P$ can be represented as the product of 
two connected spaces possessing at least two points. 
By this observation and  by Lemma \ref{lem:productinterval}, 
we have 
$\mathcal{A}(P, E_{s})=\emptyset$
for all $s\in \deltasecond$. 
Since all points in  $\Upsilon\setminus \{0_{0}\}\cup \{\, 1_{i}\mid i\in \zz_{\ge 0}\, \}$ have a neighborhood systems 
consisting of  topological  arcs, 
we have $\mathcal{A}(\Upsilon, R[\mathbf{a}])=\Upsilon$ for all 
$\mathbf{a}\in \qqcube$. 
Thus, by the definition of $D_{s, k}$, 
we have $\mathcal{A}(Z, D_{s, k})=\Upsilon$
for 
all $(s, k)\in \deltasecond\times \widehat{m}$.  
This implies that 
the metric subspace $\mathcal{A}(Z, D_{s, k})$  of $(Z, D_{s, k})$
is 
isometric to 
$(\ctree, \xi(s)\cdot 
\disdis[\topemb\circ \pi_{H,m}(s, k)])$ for 
all $(s, k)\in \deltasecond\times \widehat{m}$. 
Since $\mathcal{A}$ is isometrically invariant, 
and since $\pi_{H,m}$ is injective on 
$\deltasecond\times \widehat{m}$, 
Proposition \ref{prop:qqinj} implies that  
the condition (2) is satisfied. 
This finishes the proof. 
\end{proof}

The proof of Theorem \ref{thm:connectedemb} is 
similar to   \cite[Theorem 1.3]{Ishiki2021fractal}. 
\begin{proof}[Proof of Theorem \ref{thm:connectedemb}]
Let $\{(X_{i}, d_{i})\}_{i=1}^{n+1}$ be a sequence of metric spaces  in $\mathscr{S}$ such that 
$\grdis((X_{i}, d_{i}), (X_{j}, d_{j}))>0$ for all distinct $i, j$. 
Put $m=n+2$. 
Let 
$F: H\times \widehat{m}\to \mathscr{S}$ be a map stated in 
Proposition \ref{prop:embednm}. 
For all  $i\in \widehat{n+1}$ and 
$j\in \widehat{m}$, 
we define 
$S(i, j)=
\left\{\, s\in \deltasecond
\  \middle | \ 
F(s, j)=(X_{i}, d_{i})
\, \right\}$. 
By the conditions (1) and (2) in Proposition \ref{prop:embednm}, 
for all $i\in \widehat{n+1}$
the set $\bigcup_{j=1}^{m}S(i, j)$
is empty or a singleton. 
Thus, 
the number of pairs $(i, j)\in \widehat{n+1}\times \widehat{m}$ such that 
$\card(S(i, j))>0$ is at most $n+1$. 
 By $m=n+2$, we obtain 
$\widehat{m}\setminus \bigcup_{i=1}^{n+1}\{\, j\in \widehat{m}\mid \card(S(i, j))>0\, \}\neq \emptyset$, 
and we can take $k$ from this set. 
Then 
$\{\, (X_{i}, d_{i})\mid i=1, \dots, n+1\, \}\cap 
F\left(\deltasecond\times \{k\}\right)
=\emptyset$. 
Therefore,  the function 
$\Phi:H\to \mathscr{S}$ defined by $\Phi(s)=F(s, k)$ is injective, and hence $\Phi$ is  a  topological embedding since $H$ is compact. 
This completes the proof. 
\end{proof}

\section{The pointed Gromov--Hausdorff space}\label{sec:pointed}
In this section, we  prove Theorem \ref{thm:cat}.
In what follows, for a metric space $(X, d)$, we put $B_{d}(x, 0)=\{x\}$ and $B_{d}(x, \infty)=X$. 
We first define a rough isometry, which gives a way to 
estimate $\pgrdis$. 

Let $(X, d, a)$ and $(Y, e, b)$ be pointed metric spaces. 
Let $R\in (0, \infty]$ and 
$\epsilon \in (0, \infty)$ with $R>\epsilon$. 
A map $f: B_{d}(a, R)\to Y$ 
(or $f: X\to Y$) is said to be 
an 
\emph{$(R, \epsilon)$-rough isometry from $(X,d,  a)$ to 
$(Y, e, b)$} 
if the following conditions are satisfied:
\begin{enumerate}
\item 
$e(f(a), b)\le \epsilon$. 
\item 
$B_{e}(b, R-\epsilon)\subset N_{e}(f(B_{d}(a, R)), \epsilon)$. 
\item 
for all $x, y\in B_{d}(a, R)$, 
we have 
$|d(x, y)-e(f(x), f(y))|\le \epsilon$. 
\end{enumerate}
Remark that our definition of a rough isometry is 
the same as the definition  in \cite{herron2016gromov} 
except that  our one is limited to balls and  contains the parameter of  radius $R$. 
The following is deduced from  \cite[Lemma 3.4]{herron2016gromov}. 

\begin{lem}\label{lem:estipgh}
Let $(X, d, a)$ and $(Y, e, b)$ be  pointed metric spaces. 
Let $R\in (0, \infty]$ and $\epsilon\in (0, \infty)$ with $R>\epsilon$. 
If there exists an $(R, \epsilon)$-rough isometry 
$f: B_{d}(a, R)\to Y$ from 
$(X, d, a)$ to 
$(Y, e, b)$, 
then we have 
\[
\pgrdis((X, d, a), (Y, e, b))<
\max\left\{2\epsilon, \frac{1}{R-\epsilon}\right\}. 
\]
\end{lem}

The next proposition is an analogue of Proposition 
\ref{prop:embednm}. 
\begin{prop}\label{prop:embednmpointed}
Let $n\in \zz_{\ge 1}$ and $m\in \zz_{\ge 2}$. 
Let $H$ be a compact metrizable space, 
and $\{v_{i}\}_{i=1}^{n+1}$ be 
$n+1$ different points in $H$. 
Put $\deltasecond=H\setminus \{\, v_{i}\mid i=1, \dots, n+1\, \}$. 
Let $\{(X_{i}, d_{i}, a_{i})\}_{i=1}^{n+1}$ be a 
sequence  in $\catsettwo$ satisfying  that 
$\pgrdis((X_{i}, d_{i}, a_{i}), (X_{j}, d_{j}, a_{j}))>0$ for all distinct $i, j$. 
Then there exists a continuous map 
$F: H\times \widehat{m}\to 
\catsettwo$
such that 
\begin{enumerate}
\item 
for all $i\in \widehat{n+1}$ and $j\in \widehat{m}$ we have 
$F(v_{i}, j)=(X_{i}, d_{i}, a_{i})$;
\item 
for all $(s, i), (t, j)\in \deltasecond\times \widehat{m}$ with $(s, i)\neq (t, j)$, we have 
$F(s, i)\neq F(t, j)$. 
\end{enumerate}
\end{prop}
\begin{proof}
In what follows, 
we consider that the set $[0, \infty]$ is equipped with 
the canonical  topology homeomorphic to $[0, 1]$. 
Since every metrizable space is perfectly normal, 
and since $[0, \infty]$ is homeomorphic to $[0, 1]$, 
for each $i\in \widehat{n+1}$
we can  take a continuous function 
$\sigma_{i}:H\to [0, \infty]$ such that 
$\sigma_{i}^{-1}(0)=\{\, v_{j}\mid j\neq i\, \}$ and
$\sigma_{i}^{-1}(\infty)=\{v_{i}\}$. 
Put 
 $P(s)=\prod_{i=1}^{n+1}B_{d_{i}}(a_{i}, \sigma_{i}(s))$, and 
$p=(a_{1}, \dots, a_{n+1})\in P(s)$. 
We define a metric $E_{s}$ on $P(s)$ by 
\[
E_{s}(x, y)=
\sqrt{\sum_{i=1}^{n+1}d_{i}(x_{i}, y_{i})^{2}}. 
\]
Note that 
 $\left(P(v_{i}), E_{v_{i}}\right)$ is 
 isometric to $(X_{i}, d_{i})$. 
 Since all closed balls of a $\cat(0)$ space are
 $\cat(0)$ (see \cite[(3) in Proposition 1.4, p.160]{BH1999}), 
 by Lemma \ref{lem:prodcat}, 
 we have 
 $(P(s), E_{s}, p)\in \catsettwo$. 
 We define  a map $W:H\to \grsptwo$ by 
 \[
 W(s)=
 \begin{cases}
(X_{i}, d_{i}, a_{i}) & \text{if $s=v_{i}$ for some $i\in \widehat{n+1}$;}\\
\left(P(s), E_{s}, p\right) & \text{otherwise.}
\end{cases}
 \]
We now  prove that
the map $W:H\to \grsptwo$ is continuous. 
Take arbitrary $s\in H$ and 
$\epsilon \in (0, \infty)$. 

Case 1. ~($s\in \deltasecond$):
For a sufficient small neighborhood $V$ of 
$s$, we have $\max_{i\in \{1, \dots, n+1\}}|\sigma_{i}(s)-\sigma_{i}(t)|\le \epsilon$
for all $t\in V$. 
Since each $(X_{i}, d_{i})$ is 
a geodesic space, 
we have
\[
\hdis(B_{d_{i}}(a_{i}, \sigma_{i}(s)), B_{d_{i}}(a_{i}, \sigma_{i}(t)); X_{i}, d_{i})\le 
|\sigma_{i}(t)-\sigma_{i}(s)|.
\] 
Thus, 
for each $i\in \widehat{n+1}$ and 
for each $y\in V$, we can 
take an $(\infty, \epsilon)$-rough isometry 
$f_{i}: (B_{d_{i}}(a_{i}, \sigma_{i}(s)), d_{i}, a_{i})\to 
(B_{d_{i}}(a_{i}, \sigma_{i}(t)), d_{i}, a_{i})$. 
We define a map 
$g: P(s)\to P(t)$ by 
$g(x)=(f_{1}(x_{1}), \dots, f_{n+1}(x_{n+1}))$. 
By the triangle inequality 
for the $\ell^{2}$-norm, 
for all $x, y\in P(s)$, we obtain 
\begin{align*}
|E_{s}(x, y)-E_{t}(g(x), g(y))|
&\le
\left(\sum_{i=1}^{n+1}|d_{i}(x_{i}, y_{i})-d_{i}(f_{i}(x_{i}), f_{i}(y_{i}))|^{2}\right)^{1/2}\\
&< \sqrt{n+1}\epsilon. 
\end{align*}
Thus, 
 $g$ is an $(\infty, \sqrt{n+1}\epsilon)$-rough isometry. 
By Lemma \ref{lem:estipgh}, we have 
\[
\pgrdis((P(s), E_{s}, p), (P(t), E_{t}, p))<
2\sqrt{n+1}\epsilon.
\] 

Case 2.~($s=v_{i}$ for some $i\in \widehat{n+1}$):
Put $R=\epsilon^{-1}+\epsilon$. 
For a sufficient small neighborhood $V$ of $s (=v_{i})$, 
we have $R<\sigma_{i}(t)$ and 
$\max_{j\neq i}\sigma_{j}(t)\le \epsilon$
 for 
all $t\in V$. 
Then, for all $x, y\in P(t)$, we have 
\begin{align*}
|E_{t}(x, y)-d_{i}(x_{i}, y_{i})|\le 
\left(\sum_{j\neq i}d_{j}(x_{j}, y_{j})^{2}\right)^{1/2}
\le 2\sqrt{n}\epsilon. 
\end{align*}
Thus, the $i$-th projection 
$\pi_{i}:P(t)\to X_{i}$ is an 
$(R, 2\sqrt{n}\epsilon)$-rough isometry. 
By Lemma \ref{lem:estipgh}, we have 
\[
\pgrdis((X_{i}, d_{i}, a_{i}), (P(t), E_{t}, p))< 4\sqrt{n}\epsilon.
\] 
This finishes the proof of the continuity of $W$.

 We identify $p\in P(s)$ with $1_{0}\in \ctree$, and 
 we consider $P(s)\cap \ctree=\{p\}$. 
 We put $Z=P(s)\cup \ctree$. 
By the same way as Proposition \ref{prop:embednm}, 
we define a metric $D_{s, k}$ on $Z$. 
We also define  
$F: H\times \widehat{m}\to \catsettwo$ 
 by 
\[
F(s, k)=
\begin{cases}
(X_{i}, d_{i}, a_{i}) & \text{if $s=v_{i}$ for some $i\in \widehat{n+1}$;}\\
\left(Z, D_{s, k}, p\right) & \text{otherwise.}
\end{cases}
\]
By Proposition \ref{prop:qqconti} and 
the continuity of $W$, 
the map $F$ is continuous and the condition (1) is satisfied. 

 By the same way as Proposition \ref{prop:embednm} using the operator $\mathcal{A}$, 
 the map $F$ satisfies the condition (2). 
This finishes the proof.  
\end{proof}

\begin{proof}[Proof of Theorem \ref{thm:cat}]
By the same method as the proof of  Theorem \ref{thm:connectedemb}, 
using Proposition \ref{prop:embednmpointed} instead of 
Proposition \ref{prop:embednm}, 
we can prove Theorem \ref{thm:cat}. 
\end{proof}

\begin{rmk}
Theorem \ref{thm:connectedemb} (Proposition \ref{prop:embednm}) for the set  $\catset$ can be proven by the same method of the proof of  Theorem \ref{thm:cat} (Proposition \ref{prop:embednmpointed}); however, 
this method can not be used for $\conset$, $\pathset$, nor $\geoset$
since balls of a metric space in $\conset$, $\pathset$, or $\geoset$
do not necessarily belong to   the same class that the ambient space is in. 
In the proof of Theorem \ref{thm:connectedemb} (Proposition \ref{prop:embednm})  in 
this paper, we use  a  method by which we  can deal with all  $\conset$, $\pathset$, 
 $\geoset$, and $\catset$,  simultaneously. 
\end{rmk}

\bibliographystyle{amsplain}
\bibliography{bibtex/continua.bib}

\end{document}